\documentclass[12pt]{amsart}
\usepackage{graphicx}
\usepackage{mathrsfs}
\usepackage{epsfig}
\usepackage{amsmath}
\usepackage{amsfonts}
\usepackage{amssymb}
\usepackage{amssymb,amscd}
\usepackage{tikz}
\usepackage{verbatim}
\usepackage{booktabs}
\usepackage[all]{xy}
\usepackage{mathtools}
 \usepackage{makecell}

\usepackage{overpic}
\usepackage[hidelinks]{hyperref}
\usepackage[scr=boondox,  cal=esstix]{mathalpha}
\usepackage{adjustbox}

\usepackage{geometry}
\usepackage{cite}

\usepackage{amsthm}
\usepackage[shortlabels]{enumitem}
\usepackage{url} 
\usepackage[T1]{fontenc}
\usepackage{float}

\usepackage[utf8]{inputenc}
\usepackage[justification=centering]{caption}

\DeclareFontFamily{U}{mathx}{}
\DeclareFontShape{U}{mathx}{m}{n}{<-> mathx10}{}
\DeclareSymbolFont{mathx}{U}{mathx}{m}{n}
\DeclareMathAccent{\widehat}{0}{mathx}{"70}
\DeclareMathAccent{\widecheck}{0}{mathx}{"71}

\newtheorem{thm}{Theorem}[section]
\newtheorem{cor}[thm]{Corollary}
\newtheorem{lem}[thm]{Lemma}
\newtheorem{prop}[thm]{Proposition}

\theoremstyle{definition}
\newtheorem{defn}[thm]{Definition}
\theoremstyle{remark}
\newtheorem{rem}[thm]{Remark}
\numberwithin{equation}{section}

\newcommand{\HdC}{\mathbf H^{2}_{\mathbb C}}

\newcommand{\Pu}{\mbox{\rm PU(2,1)}}
\newcommand{\Su}{\mbox{\rm SU(2,1)}}

\renewcommand{\Re}{\mbox{\rm Re}\,}

\begin{document}
	
	\date{Oct. 24, 2026}

	\title[Delambre--Gauss formulas]{Delambre--Gauss Formulas\\in Complex Hyperbolic Space}
	\author[J. Ma]{Jiming Ma}
	\address{School of Mathematical Sciences, Fudan University, Shanghai, China}
	\email{majiming@fudan.edu.cn}
	
	\author[B. Xie]{Baohua Xie}
	\address{School of Mathematics, Hunan University, Changsha, China}
	\email{xiexbh@hnu.edu.cn}

	\keywords{Complex hyperbolic geometry,  right-angled hexagons,  Delambre-Gauss Formulas.}
	
		\subjclass[2010]{20H10, 51M09, 22E40, 51M10, 57M99.}
	
	
	\thanks{J. Ma was  supported by National Natural Science Foundation of China (No.12671086). B. Xie was supported  by Guangdong Basic and Applied Basic Research Foundation (No. 2025A1515011486).  }

\begin{abstract}
	We establish a Delambre--Gauss matrix identity for oriented right-angled
	hexagons in the complex hyperbolic plane.
	The six complex length parameters of a hexagon   determine two alternating \(3\times3\)
	matrices whose product is central in \(\mathrm{SU}(2,1)\).
	The resulting nine entry-wise equations extend the  Delambre-Gauss formula of
	Tan, Wong, and Zhang in 4-dimensional hyperbolic geometry to this complex
	hyperbolic setting.



\end{abstract}

\maketitle

\section{Introduction}

The geometry of right-angled hexagons in two- and three-dimensional hyperbolic
space is important for their applications in Fuchsian group and quasi-Fuchsian
theory, respectively.
For example, the Delambre--Gauss identities for right-angled hexagons in
3-dimensional hyperbolic space imply the sine and cosine laws for right-angled
hexagons, which can be specialized to various identities for triangles in
the hyperbolic plane \cite{TanWongZhang}.
In this paper, we study the geometry of right-angled hexagons in the complex
hyperbolic plane.
In particular, we study Delambre--Gauss identities for oriented right-angled
hexagons there, as an extension of the Tan--Wong--Zhang Delambre--Gauss
identities in 4-dimensional hyperbolic geometry to this complex hyperbolic
setting.
Our formulation uses the eigenvalues of side-preserving isometries and a
\(3\times3\) matrix closure identity.
The additional angle data in complex hyperbolic geometry already appear in the
trigonometry of triangles.
As in the case of the Delambre--Gauss identities in 4-dimensional hyperbolic
geometry, and motivated by representations of a free group into the isometry
group of 4-dimensional hyperbolic space, the authors also hope that the
Delambre--Gauss identities for right-angled hexagons in the complex hyperbolic
plane will benefit representations of a free group, and of surface groups, into
its isometry group.

 
\subsection{Triangle identities in  complex hyperbolic plane}
For the complex hyperbolic plane \(\mathbf{H}^2_{\mathbb{C}}\), we normalize
the holomorphic sectional curvature of it to \(-1\). We first consider the geometry
of triangles in \(\mathbf{H}^2_{\mathbb{C}}\), which is much more involved than
that in the hyperbolic plane \(\mathbf{H}^2_{\mathbb{R}}\).
For a noncollinear triangle in \(\mathbf{H}^2_{\mathbb{C}}\) with vertices
\(A,B,C\), let \(a,b,c\) be the opposite side lengths.
At vertex \(A\), write \(\theta_A\) for the Riemannian angle between the two
sides meeting at \(A\), \(\phi_A\) for the angle between the complex spans of
their tangent directions, and \(\mu_A\) for the angle of holomorphy of their
real span; use cyclic notation at the other vertices.
Then we have
\begin{equation}
\sin\phi_A=\sin\mu_A\sin\theta_A.
\label{eq:sin-phi-mu-theta}
\end{equation}
The definitions are recalled in Section~\ref{sec:angles}.
Goldman's two sine laws of a triangle in  \(\mathbf{H}^2_{\mathbb{C}}\) take the form~\cite[Chapter~3]{Goldman1999}
\begin{align}
\frac{\sinh(a/2)}{\sin\phi_A}
&=\frac{\sinh(b/2)}{\sin\phi_B}
=\frac{\sinh(c/2)}{\sin\phi_C},
\label{equation:firstsinelaw}\\
\frac{\sinh a}{\cos\mu_A\sin\theta_A}
&=\frac{\sinh b}{\cos\mu_B\sin\theta_B}
=\frac{\sinh c}{\cos\mu_C\sin\theta_C}.
\end{align}
The quotient expressions apply when their denominators are nonzero;
the cross-multiplied identities remain valid when a denominator vanishes.
On a totally real plane of the complex hyperbolic plane
\(\mathbf{H}^2_{\mathbb{C}}\), whose sectional curvature is \(-1/4\), the first
identity is the ordinary sine law with side lengths \(a/2,b/2,c/2\).
On a complex geodesic of \(\mathbf{H}^2_{\mathbb{C}}\), whose sectional
curvature is \(-1\), the second identity becomes
\begin{equation}\label{equation:sinelaw}
\frac{\sinh a}{\sin\theta_A}
=\frac{\sinh b}{\sin\theta_B}
=\frac{\sinh c}{\sin\theta_C}.
\end{equation}
The corresponding cosine law~\cite[Chapter~3]{Goldman1999} can be written
in either of the following equivalent forms:
\begin{align}
&\cosh^2(c/2)
-\bigl(\cosh(a/2)\cosh(b/2)
-\cos\theta_C\sinh(a/2)\sinh(b/2)\bigr)^2 \notag\\
&\hspace{18mm}
=\cos^2\mu_C\sin^2\theta_C\sinh^2(a/2)\sinh^2(b/2),
\label{eq:goldman-cosine-one}\\
&\cosh a\cosh b-\cos\theta_C\sinh a\sinh b-\cosh c \notag\\
&\hspace{18mm}
=2\sin^2\mu_C\sin^2\theta_C\sinh^2(a/2)\sinh^2(b/2).
\label{eq:goldman-cosine-two}
\end{align}
At a Riemannian right angle, \(\theta_C=\pi/2\), these give the
Pythagorean identities
\begin{align}
\cosh^2(c/2)
&=\cosh^2(a/2)\cosh^2(b/2)
+\cos^2\mu_C\sinh^2(a/2)\sinh^2(b/2),\\
\cosh c
&=\cosh a\cosh b-2\sin^2\mu_C\sinh^2(a/2)\sinh^2(b/2).
\end{align}
These formulas provide the triangular background for our hexagon
identities in complex  hyperbolic geometry: even at a Riemannian right angle, the complex angle remains
relevant. Therefore, the Delambre--Gauss identities for hexagons in complex  hyperbolic plane are inevitably more involved than those in the real hyperbolic setting.

\subsection{Delambre--Gauss identities for hexagons in real hyperbolic space}

Right-angled hexagons are basic building blocks in real hyperbolic
geometry and in the construction of Fenchel--Nielsen coordinates of hyperbolic surfaces;
see~\cite{Fenchel}. Complex length coordinates in the quasi-Fuchsian
setting were developed by Kourouniotis~\cite{Kourouniotis1994} and
Tan~\cite{Tan1994}. Related coordinate constructions in complex
hyperbolic geometry and for \(\mathrm{SL}(3,\mathbb C)\) representations
appear in~\cite{ParkerPlatis2008} and~\cite{FigueroaParker2025},
respectively.

For an oriented right-angled hexagon in real hyperbolic three-space,
write \(l_j\) for its signed complex side lengths in the convention of
\cite[Appendix~A]{TanWongZhang}.
The Delambre--Gauss formulas are
\begin{align}
\cosh\frac{1}{2}(l_1 + l_3) \cosh\frac{1}{2}l_2
&= \epsilon \cosh\frac{1}{2}(l_4 + l_6) \cosh\frac{1}{2}l_5, \\
-\sinh\frac{1}{2}(l_1 + l_3) \cosh\frac{1}{2}l_2
&= \epsilon \cosh\frac{1}{2}(l_4 - l_6) \sinh\frac{1}{2}l_5, \\
-\cosh\frac{1}{2}(l_1 - l_3) \sinh\frac{1}{2}l_2
&= \epsilon \sinh\frac{1}{2}(l_4 + l_6) \cosh\frac{1}{2}l_5, \\
\sinh\frac{1}{2}(l_1 - l_3) \sinh\frac{1}{2}l_2
&= \epsilon \sinh\frac{1}{2}(l_4 - l_6) \sinh\frac{1}{2}l_5.
\end{align}
Here \(\epsilon\in\{1,-1\}\) records the choices of lifts of the identity  from  \(\mathrm{PSL}(2,\mathbb C)\) to  \(\mathrm{SL}(2,\mathbb C)\).
These formulas can be recast in matrix form. Define
\(A_n \in \mathrm{SL}(2,\mathbb C)\), \(n=1,\ldots,6\), by
\begin{align*}
A_n&=\begin{pmatrix}
e^{l_n/2}&0\\
0&e^{-l_n/2}
\end{pmatrix}, \qquad n=1,3,5; \\
A_n&=\begin{pmatrix}
\cosh(l_n/2)&\sinh(l_n/2)\\
\sinh(l_n/2)&\cosh(l_n/2)
\end{pmatrix}, \qquad n=2,4,6.
\end{align*}
Then the Delambre--Gauss formulas follow from the matrix identity
\[
A_1A_2A_3=\epsilon A_6^{-1}A_5^{-1}A_4^{-1},
\]
where \(\epsilon^2=1\).

Tan, Wong, and Zhang also obtained Delambre--Gauss formulas for
augmented right-angled hexagons in real hyperbolic four-space
\cite{TanWongZhang}. Their proof follows the side isometries around the
hexagon and expresses closure as a matrix identity. This inspired our proof of the Delambre--Gauss identities in the complex hyperbolic plane.

\subsection{Delambre--Gauss identities in the complex hyperbolic plane}

In this paper, we consider right-angled hexagons \((\vec{L}_i)_{i=1}^6\) in
\(\mathbf{H}^2_{\mathbb{C}}\) such that the real oriented geodesics \(\vec{L}_i\) and
\(\vec{L}_{i+1}\), with indices taken modulo \(6\), meet at a Riemannian right angle,
and the angle between the complex spans of \(\vec{L}_i\) and \(\vec{L}_{i+1}\) is the
common complex angle \(\theta\). The moduli space of such right-angled
hexagons is the space \(\mathcal{H}^{\theta}\) in
Subsection~\ref{subsec:moduli}. Here \(\mathcal{H}^{0}\) is the well-studied space 
of right-angled hexagons in \(\mathbf{H}^2_{\mathbb{R}}\), 
while \(\mathcal{H}^{\pi/2}\) is the well-studied space of right-angled complex hexagons in \(\mathbf{H}^2_{\mathbb{C}}\); 
these spaces have dimensions \(3\) and \(4\), respectively.


Write
\(x_j=\vec{L}_j\cap \vec{L}_{j+1}\), with indices modulo six, and let \(u_j\)
be the positive unit tangent vector field on the oriented complete line \(\vec{L}_j\).
At the vertex $x_j$, orthogonality of  \(\vec{L}_j\) and \(\vec{L}_{j+1}\) implies that
\(\langle u_j,u_{j+1}\rangle_{x_j}\) is purely imaginary.
 Its modulus is
\(\cos\theta\), but its sign also matters. The side map preserving
\(\vec{L}_j\) and carrying oriented \(\vec{L}_{j-1}\) to oriented \(\vec{L}_{j+1}\)
requires consecutive Hermitian products to have opposite signs.
Accordingly, we use the following signed form of the angle hypothesis:
\begin{equation}\label{eq:signed-angle}
\langle u_j,u_{j+1}\rangle_{x_j}
=(-1)^{j-1} i  \cos\theta ,\qquad j=1,\ldots,6.
\end{equation}
Its role in the existence and uniqueness of the side maps is explained
in Lemma~\ref{thm:closure}.

\begin{thm}[Delambre--Gauss identity with an angle parameter]
\label{thm:main}
Let \((\vec{L}_1,\dots,\vec{L}_6) \in  \mathcal{H}^{\theta}\) be an oriented right-angled hexagon with
\(x_{j-1}\ne x_j\) for every \(j\), common complex angle
\( \theta \in (0, \pi/2]\), and orientations satisfying
\eqref{eq:signed-angle}. Let \(\lambda_j\in\mathbb C^*\) be the
complex length parameters of its side maps, defined in
Subsection~\ref{subsection:sideparameters}. Then
\begin{equation}\label{eq:transport-main}
\begin{split}
&A_1(\theta,\lambda_1)\,A_2(\theta,\lambda_2)\,A_3(\theta,\lambda_3)
\\&\quad = 
\varepsilon\,
A_6(\theta,\lambda_6^{-1})\,A_5(\theta,\lambda_5^{-1})\,A_4(\theta,\lambda_4^{-1}),
\qquad \varepsilon^3=1.
\end{split}
\end{equation}
Where \(A_i(\theta, \lambda^{\pm}_i)\) are \(3\times3\) complex matrices, and  then the nine scalar entries of \eqref{eq:transport-main} are the
Delambre--Gauss identities for oriented right-angled hexagons in
\(\mathbf H^2_{\mathbb C}\).
\end{thm}

Equation~\eqref{eq:transport-main} is a set of identities relating
the intrinsic geometric quantities \(\theta\) and \(\{\lambda_n\}^{6}_{n=1}\) of a
right-angled hexagon. Therefore, Theorem~\ref{thm:main} is a
geometric theorem about right-angled hexagons.

The matrices in Theorem \ref{thm:main} are
\[
A_n(\theta,\lambda)=
\begin{cases}
\operatorname{diag}(\lambda,\bar\lambda/\lambda,1/\bar\lambda),
   & n=1,3,5,\\
Q\,\operatorname{diag}(\lambda,\bar\lambda/\lambda,1/\bar\lambda)\,Q^{-1},
   & n=2,4,6,
\end{cases}
\]
where \(Q=Q(\theta)\) is displayed in \eqref{eq:Q}.
The scalar \(\varepsilon\) records the choices of lifts of the identity  from \(\Pu\) to \(\Su\).
The nine entry equations are not asserted to be independent.

Two features distinguish this formulation from its real hyperbolic
counterparts. First, in the \(2\times2\) setting of  three-dimensional  real geometry  a complex displacement
\(\ell+i\phi\) is represented by eigenvalues \(\lambda,\lambda^{-1}\),
where \(\lambda=e^{(\ell+i\phi)/2}\). The identities
\[
\lambda=\cosh\frac{\ell+i\phi}{2}+\sinh\frac{\ell+i\phi}{2},
\qquad
\lambda^{-1}=\cosh\frac{\ell+i\phi}{2}-\sinh\frac{\ell+i\phi}{2}
\]
lead to the usual hyperbolic sine and cosine expressions. In \(\Su\),
the corresponding eigenvalues are
\[
\lambda,\quad\bar\lambda/\lambda,\quad1/\bar\lambda,
\qquad \lambda=e^{\ell/2-i\phi/3};
\]
see~\cite{Parker2012}. Both \(\lambda\) and \(\bar\lambda\) therefore
enter the matrix identity, and eigenvalues provide convenient variables.

Second, the \(2\times2\) formulations in real hyperbolic dimensions two
and three use real and complex matrices, respectively, while the
four-dimensional formulation in~\cite{TanWongZhang} uses quaternionic
matrices. Each matrix equality has four entries. Our \(3\times3\)
equality has nine entries, but this is a distinction between matrix
representations, not a count of independent equations. The Lorentz
description of the real hyperbolic plane also uses \(3\times3\)
matrices; Section~\ref{sec:real-hyperbolic-degeneration} explains its
relation to the \(2\times2\) description, in particular
Corollary~\ref{cor:real-case}.

The paper is organized as follows. We first recall the models, metric, and
angle conventions for complex hyperbolic space in
Sections~\ref{sec:models} and~\ref{sec:angles}. We then introduce oriented
right-angled hexagons and derive the matrix identity by following the side
isometries around the hexagon in Sections~\ref{sec:hexagons} and~\ref{sec:transport}.
In Section~\ref{sec:real-hyperbolic-degeneration}, we distinguish two
geometrically different subcases, namely the real plane case and the complex
geodesic case. Section~\ref{sec:mostow-canonical-hexagon} concludes
with the matrix verification for Mostow's canonical example.

\textbf{AI Declaration.} The authors used an AI tool only for English language editing and grammar correction. No mathematical arguments, proofs, or computations were generated or checked by AI. All mathematical content and results are the authors' own.

\section{Complex hyperbolic space}\label{sec:models}

In this section we collect the basic facts about complex hyperbolic space
\(\mathbf H^2_{\mathbb C}\) that will be used later. We begin with the
Hermitian vector space \(\mathbb C^{2,1}\), then describe the ball and Siegel
domain models, the Bergman metric, the isometry group, and finally the
parametrization of geodesics. See \cite{Goldman1999,Par2003}. Although we have only define the complex hyperbolic space here, the real hyperbolic space is defined analogously, by restricting each quantity from complex to real.

\subsection{The Hermitian vector space \texorpdfstring{\(\mathbb C^{2,1}\)}{C\^{}\{2,1\}}}

Consider \(\mathbb C^{2,1}\), the complex vector space \(\mathbb C^3\) equipped
with a Hermitian form \(\langle\cdot,\cdot\rangle\) of signature \((2,1)\).
For column vectors
\[
\mathbf z=(z_1,z_2,z_3)^{\top},\qquad
\mathbf w=(w_1,w_2,w_3)^{\top}
\]
in \(\mathbb C^3\), a general Hermitian form of signature \((2,1)\) is given by
\[
\langle \mathbf z,\mathbf w\rangle=\mathbf w^*H\mathbf z,
\]
where \(H\) is a nonsingular \(3\times3\) Hermitian matrix of signature
\((2,1)\). The form is linear in the first variable and conjugate linear in the
second.

Different choices of \(H\) of the same signature give isomorphic models; the
choice of \(H\) is essentially a choice of basis. In what follows we will
introduce the particular Hermitian forms that are most convenient for the ball
model and the Siegel domain model.

Since \(\langle \mathbf z,\mathbf z\rangle\) is real for every \(\mathbf z\), we may
define the subsets
\begin{align*}
V_-&=\{\mathbf z\in\mathbb C^{2,1}:\langle \mathbf z,\mathbf z\rangle<0\},\\
V_0&=\{\mathbf z\in\mathbb C^{2,1}\setminus\{\mathbf 0\}:
\langle \mathbf z,\mathbf z\rangle=0\},\\
V_+&=\{\mathbf z\in\mathbb C^{2,1}:\langle \mathbf z,\mathbf z\rangle>0\}.
\end{align*}
Because
\(
\langle \lambda \mathbf z,\lambda \mathbf z\rangle
=
|\lambda|^2\langle \mathbf z,\mathbf z\rangle,
\)
the sign of \(\langle \mathbf z,\mathbf z\rangle\) is well defined on projective
classes.

Let
\(
\mathbb P:\mathbb C^{2,1}\setminus\{\mathbf 0\}\longrightarrow \mathbb{CP}^2
\)
be the standard projectivization. In the affine chart \(z_3\neq 0\),
\[
\mathbb P(z_1,z_2,z_3)
=
\left(\frac{z_1}{z_3},\frac{z_2}{z_3}\right).
\]
Complex hyperbolic \(2\)-space is defined by
\(
\mathbf H^2_{\mathbb C}=\mathbb P V_-,
\)
and its boundary is
\(
\partial \mathbf H^2_{\mathbb C}=\mathbb P V_0.
\)

\subsection{The ball model of the complex hyperbolic plane}

Take the Hermitian form
\[
H_B=\operatorname{diag}(1,1,-1).
\]
For \(z=(z_1,z_2)\in\mathbb C^2\), define its standard lift to \(\mathbb C^{2,1}\)
by
\[
\mathbf z=
\begin{pmatrix}
z_1\\
z_2\\
1
\end{pmatrix}.
\]
Then
\(
\langle \mathbf z,\mathbf z\rangle
=
|z_1|^2+|z_2|^2-1.
\)
Hence
\[
\mathbf H^2_{\mathbb C}
\cong
\mathbb B^2
=
\{(z_1,z_2)\in\mathbb C^2:|z_1|^2+|z_2|^2<1\},
\]
and
\[
\partial \mathbf H^2_{\mathbb C}
\cong
S^3
=
\{(z_1,z_2)\in\mathbb C^2:|z_1|^2+|z_2|^2=1\}.
\]
This is the ball model of \(\mathbf H^2_{\mathbb C}\), the complex analogue of the
Poincar\'e disc model of the hyperbolic plane.

\subsection{The Siegel domain model of the complex hyperbolic plane}

Take the Hermitian form
\[
H_S=
\begin{pmatrix}
0&0&1\\
0&1&0\\
1&0&0
\end{pmatrix}.
\]
For \(z=(z_1,z_2)\in\mathbb C^2\), again take the standard lift
\[
\mathbf z=
\begin{pmatrix}
z_1\\
z_2\\
1
\end{pmatrix}.
\]
Then
\(
\langle \mathbf z,\mathbf z\rangle
=
2\Re(z_1)+|z_2|^2.
\)
Therefore
\[
\mathbf H^2_{\mathbb C}
=
\{(z_1,z_2)\in\mathbb C^2:2\Re(z_1)+|z_2|^2<0\}.
\]
The part of the boundary lying in \(\mathbb C^2\) is
\[
\partial \mathbf H^2_{\mathbb C}\cap\mathbb C^2
=
\{(z_1,z_2)\in\mathbb C^2:2\Re(z_1)+|z_2|^2=0\}.
\]
One must add the point \(\infty\), whose standard lift is
\[
\begin{pmatrix}
1\\
0\\
0
\end{pmatrix}.
\]
This is the Siegel domain model, the complex analogue of the upper half-plane
model of the hyperbolic plane.

The Cayley transform from the Siegel model to the ball model is
\[
C=
\frac{1}{\sqrt{2}}\begin{pmatrix}
1&0&1\\
0&\sqrt{2}&0\\
1&0&-1
\end{pmatrix}.
\]

\subsection{Distance and the Bergman metric}

The distance function \(\rho(\cdot,\cdot)\) on \(\mathbf H^2_{\mathbb C}\) is
defined by
\[
\cosh^2\left(\frac{\rho(z,w)}{2}\right)
=
\frac{\langle \mathbf z,\mathbf w\rangle\langle \mathbf w,\mathbf z\rangle}
{\langle \mathbf z,\mathbf z\rangle\langle \mathbf w,\mathbf w\rangle},
\]
where \(\mathbf z\) and \(\mathbf w\) are arbitrary lifts of \(z\) and \(w\). This
expression is independent of the chosen lifts.

Equivalently, the Bergman metric is given in differential form by
\[
ds^2
=
-\frac{4}{\langle \mathbf z,\mathbf z\rangle^2}
\det
\begin{pmatrix}
\langle \mathbf z,\mathbf z\rangle & \langle \mathbf z,d\mathbf z\rangle\\
\langle d\mathbf z,\mathbf z\rangle & \langle d\mathbf z,d\mathbf z\rangle
\end{pmatrix}.
\]
The constant \(4\) normalizes the holomorphic sectional curvature of
\(\mathbf H^2_{\mathbb C}\) to be \(-1\).

\subsection{Isometries}

The holomorphic isometry group of \(\mathbf H^2_{\mathbb C}\) is the projective
unitary group \(\mathrm{PU}(2,1)\). It is often convenient to lift to its
triple cover, namely the special unitary group \(\mathrm{SU}(2,1)\), so that
\[
\mathrm{PU}(2,1)
=
\mathrm{SU}(2,1)/\{I,\omega I,\omega^2 I\},
\]
where
\[
\omega=\frac{-1+i\sqrt3}{2}
\]
is a cube root of unity. This is analogous to the double cover
\(\mathrm{SL}(2,\mathbb C)\to\mathrm{PSL}(2,\mathbb C)\) in three dimensional  real hyperbolic
geometry.

\subsection{Complex lines and polar vectors}

A complex line in \(\mathbf H^2_{\mathbb C}\) is the projectivization of a
complex two-dimensional subspace of \(\mathbb C^{2,1}\) on which the Hermitian
form restricts to a form of signature \((1,1)\).

\begin{defn}\label{def:polar}
	
Suppose \(L\) is a complex two-dimensional subspace of \(\mathbb C^{2,1}\)
with signature \((1,1)\). A non-zero vector \(\mathbf n\in\mathbb C^{2,1}\)
is said to be a \emph{polar vector} for \(L\) if it is orthogonal to \(L\)
with respect to the Hermitian form. By a slight abuse of notation, we also
write \(L\) for its projectivization in \(\mathbf H^2_{\mathbb C}\), which is
a complex line.
\end{defn}

The following proposition, due to Goldman, describes the relationship between
two complex lines in terms of their polar vectors.

\begin{prop}[{\cite{Goldman1999}}]\label{thm:polar-formula}
Let \(L\) and \(M\) be complex lines in \(\mathbf H^2_{\mathbb C}\) with polar vectors \(\mathbf n,\mathbf m\)
respectively, and set
\[
N(L,M)
=
\frac{\langle \mathbf n,\mathbf m\rangle
\langle \mathbf m,\mathbf n\rangle}
{\langle \mathbf n,\mathbf n\rangle
\langle \mathbf m,\mathbf m\rangle}.
\]
Then:
\begin{itemize}
\item If \(N(L,M)>1\), then \(L\) and \(M\) are ultraparallel and
\[
\cosh^2\left(\frac{\rho(L,M)}{2}\right)=N(L,M).
\]
\item If \(N(L,M)=1\), then \(L\) and \(M\) are asymptotic or coincide.
\item If \(N(L,M)<1\), then \(L\) and \(M\) intersect and
\[
\cos^2\theta=N(L,M),
\]
where \(\theta\) is the angle of intersection between \(L\) and \(M\).
\end{itemize}
\end{prop}

\subsection{Real geodesics}

Consider a pair of distinct null vectors \(\mathbf p,\mathbf q\in V_0\).
Without loss of generality, normalize them so that
\[
\langle {\mathbf p},\mathbf q\rangle=-1.
\]
These vectors correspond to a pair of points \(p\) and \(q\) on
\(\partial\mathbf H^2_{\mathbb C}\). The geodesic with endpoints \(p\) and
\(q\) admits the following explicit parametrization.

\begin{prop}
Let \(\mathbf p,\mathbf q\in V_0\) be null vectors with
\(\langle {\mathbf p},\mathbf q\rangle=-1\). For all real \(t\), let
\(\gamma(t)\) be the point in \(\mathbf H^2_{\mathbb C}\) corresponding to the
vector
\[
e^{t/2}\mathbf p+e^{-t/2}\mathbf q
\]
in \(\mathbb C^{2,1}\). Then
\[
\gamma=\{\gamma(t)\mid t\in\mathbb R\}
\]
is the geodesic in \(\mathbf H^2_{\mathbb C}\) with endpoints \(p\) and \(q\),
parametrized by arc length.
\end{prop}

\subsection{Eigenvalues and complex displacement for loxodromic maps}\label{subsection:eigenvalues}

In this subsection we recall how the trace of a loxodromic isometry in
\(\mathrm{SU}(2,1)\) encodes its translation length and rotational part as it
acts on \(\mathbf{H}^2_{\mathbb{C}}\). We begin with the familiar cases of
\(\mathrm{SL}(2,\mathbb{R})\) and \(\mathrm{SL}(2,\mathbb{C})\), in
\(\mathbf{H}^2_{\mathbb{R}}\) and \(\mathbf{H}^3_{\mathbb{R}}\), respectively,
and then discuss the corresponding statement for \(\mathrm{SU}(2,1)\) acting
on \(\mathbf{H}^2_{\mathbb{C}}\). See \cite{Parker2012}.

A loxodromic element \(A\in\mathrm{SL}(2,\mathbb R)\) or \(A\in\mathrm{SU}(1,1)\)
with eigenvalues \(\lambda\) and \(\lambda^{-1}\), where \(|\lambda|>1\),
corresponds to a hyperbolic isometry, which we also denote by \(A\), in
\(\mathrm{PSL}(2,\mathbb R)\) or \(\mathrm{PU}(1,1)\), respectively. Since \(A\)
is loxodromic, it has two fixed points on the boundary of the hyperbolic plane,
and these are the projections of the eigenspaces. The geodesic  in the hyperbolic plane joining these
two fixed points is called the \emph{axis} of \(A\), and is denoted
\(\tilde{\alpha}\).

Then  \(\mathbf H^2_{\mathbb R}/\langle A\rangle\), or \(\mathbf H^1_{\mathbb C}/\langle A\rangle\), is a hyperbolic cylinder, and
\(\alpha=\tilde{\alpha}/\langle A\rangle\) is the hyperbolic geodesic around its
waist with hyperbolic length \(\ell\), where
\[
|\lambda|=e^{\ell/2},
\qquad
|\operatorname{tr}(A)|=2\cosh(\ell/2).
\]
In other words, \(A\) translates along its axis by a hyperbolic translation
length of \(\ell\).

Similarly, when \(A\in\mathrm{SL}(2,\mathbb C)\), its trace corresponds to a
complex length. More precisely, suppose
\[
\operatorname{tr}(A)=\lambda+\lambda^{-1},
\qquad |\lambda|>1.
\]
Then again \(|\lambda|=e^{\ell/2}\).
To find the argument of \(\lambda\), fix \(z\in\tilde{\alpha}\subset
\mathbf{H}^3_{\mathbb{R}}\), and let \(\xi\in T_z(\mathbf{H}^3_{\mathbb{R}})\)
be a tangent vector orthogonal to \(\tilde{\alpha}\), the axis of \(A\). Let
\(\xi'\) denote the parallel transport of \(\xi\) along \(\tilde{\alpha}\) to
\(A(z)\). Then
\[
A_*\xi=\xi' e^{i\phi},
\]
where \(e^{i\phi}\) is understood after identifying
\(T_{A(z)}(\mathbf{H}^3_{\mathbb{R}})\) with \(T_z(\mathbf{H}^3_{\mathbb{R}})\)
via this parallel transport.
In other
words, \(A\) translates along \(\tilde{\alpha}\) by a hyperbolic distance
\(\ell\) and rotates the tangent space by an angle \(\phi\). Then
\[
\lambda=e^{\ell/2+i\phi/2},
\qquad
\operatorname{tr}(A)=2\cosh(\ell/2+i\phi/2).
\]
Since \(\phi\) is defined modulo \(2\pi\), the imaginary part of
\(\ell/2+i\phi/2\) is defined modulo \(\pi\). This introduces an ambiguity of
\(\pm1\) in the trace, and this corresponds exactly to the ambiguity introduced
when lifting \(A\) from \(\mathrm{PSL}(2,\mathbb C)\) to
\(\mathrm{SL}(2,\mathbb C)\).


We now illustrate how the geometric action of \(A\in\mathrm{SU}(2,1)\)  on \(\mathbf H^2_{\mathbb C}\) is
recorded by \(\operatorname{tr}(A)\). In principle the relationship is very
similar to the cases of \(\mathrm{SL}(2,\mathbb R)\) and
\(\mathrm{SL}(2,\mathbb C)\), but the functions involved are more complicated.

\begin{prop}\label{prop:loxodromic-su21}
Let \(A\in\mathrm{SU}(2,1)\) be a loxodromic map with axis \(\tilde{\alpha} \subset  \mathbf H^2_{\mathbb C}\).
Let \(\lambda\in\mathbb C\) be the eigenvalue of \(A\) with \(|\lambda|>1\).
Suppose that \(A\) has Bergman translation length \(\ell\) along
\(\tilde{\alpha}\) and rotates the complex normal direction to the complex geodesic
containing \(\tilde{\alpha}\) by an angle \(\phi\), measured relative
to parallel translation along the axis. Then
\[
\lambda=e^{\ell/2-i\phi/3},
\]
and
\[
\operatorname{tr}(A)
=
2\cosh(\ell/2)e^{-i\phi/3}+e^{2i\phi/3}.
\]
Furthermore, since \(\phi\) is defined modulo \(2\pi\), the argument of
\(\lambda\) is only given modulo \(2\pi/3\), and so these formulae are only well
defined on a chosen lift to \(\mathrm{SU}(2,1)\). On \(\mathrm{PU}(2,1)\),
\(\lambda\) is determined up to a cube root of unity.
We call \(\lambda\) the \emph{complex length parameter}, to distinguish
it from the additive displacement \(\ell+i\phi\).
See Figure~\ref{fig:complex-length}.
\end{prop}

\begin{figure}[H]
    \centering
    \tikzset{every picture/.style={line width=0.6pt}} 

    \begin{tikzpicture}[x=0.6pt, y=0.6pt, yscale=-1, xscale=1]

    \draw [line width=1.5] (202.63,235.06) .. controls (202.63,213.55) and (240.88,196.12) .. (288.07,196.12)
        .. controls (335.25,196.12) and (373.5,213.55) .. (373.5,235.06)
        .. controls (373.5,256.57) and (335.25,274) .. (288.07,274)
        .. controls (240.88,274) and (202.63,256.57) .. (202.63,235.06) -- cycle;

    \draw [line width=0.75] (296.5,18) -- (296.52,23) -- (297.5,302);

    \draw [line width=1.5] (247.94,236.8) -- (297.06,237.96);

    \draw [line width=1.5] [dash pattern={on 1.69pt off 2.76pt}] (248.63,75.22) -- (297.75,76.39);

    \draw [line width=1.5] (259.36,98.47) -- (297.75,76.39);

    \draw [line width=1.5] (199.86,83.94) .. controls (199.86,62.43) and (238.11,45) .. (285.3,45)
        .. controls (332.48,45) and (370.73,62.43) .. (370.73,83.94)
        .. controls (370.73,105.45) and (332.48,122.88) .. (285.3,122.88)
        .. controls (238.11,122.88) and (199.86,105.45) .. (199.86,83.94) -- cycle;

    \draw [draw opacity=0][line width=1.5]
        (275.44,86.31) .. controls (271.81,84.8) and (269.55,82.28) .. (269.95,79.7)
        .. controls (270.22,77.92) and (271.71,76.44) .. (273.89,75.46)
        -- (281.61,80.96) -- cycle;
    \draw [line width=1.5]
        (275.44,86.31) .. controls (271.81,84.8) and (269.55,82.28) .. (269.95,79.7)
        .. controls (270.22,77.92) and (271.71,76.44) .. (273.89,75.46);

    \draw (297.5,171) -- (297.5,156);
    \draw [shift={(297.5,154)}, rotate=90]
        [color={rgb,255:red,0; green,0; blue,0}][line width=0.75]
        (10.93,-3.29) .. controls (6.95,-1.4) and (3.31,-0.3) .. (0,0)
        .. controls (3.31,0.3) and (6.95,1.4) .. (10.93,3.29);

    \draw (247.82,74.6) node [anchor=north west][inner sep=0.75pt] {$\phi$};
    \draw (310.08,146.67) node [anchor=north west][inner sep=0.75pt] {$\ell$};

    \end{tikzpicture}

    \caption{Translation of \(A\in\mathrm{SU}(2,1)\) along the axis \(\tilde{\alpha} \subset  \mathbf H^2_{\mathbb C}\) and rotation in its complex normal direction}
    \label{fig:complex-length}
\end{figure}
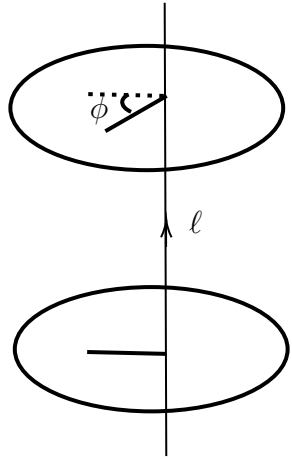

\section{Angles in complex hyperbolic space}\label{sec:angles}

We recall the angle conventions used in the introduction;
see~\cite[Chapter~2]{Goldman1999}. Let \(u,v\) be real-linearly
independent tangent vectors at \(p \in  \mathbf H^2_{\mathbb C}\), and put
\(P=\operatorname{span}_{\mathbb R}\{u,v\}\).
The complex structure \(J\)  on  \(\mathbf H^2_{\mathbb C}\) satisfies \(J^2=-\operatorname{Id}\).
The tangent Hermitian product is linear in its first variable, with
\(g(u,v)=\operatorname{Re}\langle u,v\rangle\), where \(g\) is the
Bergman metric. The angles defined in Subsection  \ref{subsection:angle} are unchanged by positive real
rescaling, so we take \(u,v\) to be unit vectors.

\subsection{The three angles} \label{subsection:angle}
For \(u,v\) be vectors at \(p \) as above, 
the \emph{Riemannian angle} \(\theta\in(0,\pi)\) of \(u\) and \(v\) is the angle
between the two real tangent directions. The \emph{complex angle}
\(\phi\in[0,\pi/2]\)  of \(u\) and \(v\) is the angle between their complex spans.
The \emph{angle of holomorphy} \(\mu\in[0,\pi/2]\)  of \(u\) and \(v\) is associated
to the real plane \(P\). Their formulas are
\begin{equation}\label{eq:angle-definitions}
\cos\theta=\operatorname{Re}\langle u,v\rangle,\qquad
\cos\phi=|\langle u,v\rangle|,\qquad
\cos\mu=\frac{|\operatorname{Im}\langle u,v\rangle|}{\sin\theta}.
\end{equation}
In particular, \(\phi=0\) exactly when \(u,v\) are complex-linearly
dependent, and \(\phi=\pi/2\) exactly when they are Hermitian
orthogonal. The complex angle is also the angle between the
complex geodesics tangent to \(\mathbb C u\) and \(\mathbb C v\).

To interpret the angle  \(\mu\), orthonormalize the real pair by setting
\[
e_1=u,\qquad e_2=\frac{v-\cos\theta\,u}{\sin\theta}.
\]
Since \(\langle u,Jv\rangle=-i\langle u,v\rangle\), we have
\[
\cos\mu=|g(e_1,Je_2)|
       =\frac{|g(u,Jv)|}{\sin\theta}.
\]
Thus \(\mu=0\) means \(P=JP\), while \(\mu=\pi/2\) means
\(P\perp JP\), the totally real case. This description depends only on \(P\), rather than on the particular basis
\(\{u, v\}\).

\subsection{The fundamental relation}

Separating the real and imaginary parts of the Hermitian product $\langle u,v\rangle$ gives
\[
\cos^2\phi=\cos^2\theta+\cos^2\mu\,\sin^2\theta.
\]
Consequently,
\begin{equation}\label{eq:angle-relation}
\sin\phi=\sin\mu\sin\theta.
\end{equation}
The nonnegative square roots are determined by the angle ranges.
The complex angle therefore depends on both the Riemannian angle
between the directions and the angle of holomorphy of their real span.

\subsection{Normal form and the sign of the Hermitian product}

After a unitary change of tangent coordinates, we may assume the pair has the
form
\[
u=\begin{pmatrix}1\\0\end{pmatrix},\qquad
v=\begin{pmatrix}e^{i\eta}\cos\phi\\ \sin\phi\end{pmatrix},
\]
where \(\eta\in(-\pi,\pi]\) and
\(\langle u,v\rangle=e^{-i\eta}\cos\phi\).
Indeed, first send \(u\) to the first coordinate vector, and then multiply the
complex orthogonal complement of \(u\) by a unit complex number so that the
second component of \(v\) is nonnegative. Since \(v\) has unit norm, this
component equals \(\sin\phi\), while the first component has modulus
\(\cos\phi\) and hence is \(e^{i\eta}\cos\phi\). When \(\cos\phi=0\), the phase
is irrelevant and we set \(\eta=0\).

When the projection of \(v\) onto \(\mathbb C u\) is nonzero,
the unsigned angle between \(u\) and that projection is
\(|\eta|\), Goldman's second pseudo-angle. The signed phase \(\eta\)
also records the sign of the imaginary part. In this normal form,
\[
\cos\theta=\cos\phi\cos\eta,
\]
and
\[
\begin{aligned}
\cos\mu\sin\theta
 &=\cos\phi\,|\sin\eta|\\
 &=\sqrt{\cos^2\phi-\cos^2\theta}
  =\sqrt{\sin^2\theta-\sin^2\phi}.
\end{aligned}
\]
The absolute value is necessary because \(\mu\) is unsigned.
For oriented sides of a right-angled hexagon, the sign of the Hermitian product must be
specified separately, as in \eqref{eq:signed-angle}.

\subsection{Special cases and curvature}

\paragraph{A complex plane.}
If \(P=\operatorname{span}_{\mathbb R}\{u,v\}\) is complex, then \(\mu=\phi=0\).
The Riemannian angle may vary, but both real geodesics lie in
the same complex geodesic.

\paragraph{A totally real plane.}
If \(P=\operatorname{span}_{\mathbb R}\{u,v\}\) is totally real, then \(\mu=\pi/2\), so
\[
\sin\phi=\sin\theta,\qquad
\phi=\min\{\theta,\pi-\theta\}.
\]
Thus \(\phi=\theta\) for an acute or right Riemannian angle.
For an obtuse angle, the complex angle is its supplement.

\paragraph{Orthogonality.}
If \(\langle u,v\rangle=0\), then
\(\phi=\mu=\theta=\pi/2\).
Conversely, at a Riemannian right angle, \eqref{eq:angle-relation}
gives \(\phi=\mu\), which may take any value in \([0,\pi/2]\).
Thus Riemannian orthogonality alone does not imply complex
orthogonality.

With holomorphic sectional curvature normalized to \(-1\), the
sectional curvature of \(P\) is
\[
\kappa(P)=-\frac{1+3\cos^2\mu}{4}.
\]
A complex geodesic has curvature \(-1\), whereas a totally real,
totally geodesic plane has curvature \(-1/4\).
In real hyperbolic geometry the sectional curvature is constant;
here it depends on the tangent plane through \(\mu\).
The sine, cosine, and Pythagorean laws in the introduction reflect
this distinction. At the vertices of a triangle, these same definitions
give \((\theta_A,\phi_A,\mu_A)\) and their cyclic counterparts.

\section{Oriented right-angled hexagons in complex hyperbolic space}
\label{sec:hexagons}

From this section onward, \(\theta\) denotes the common complex angle
of consecutive sides of a right-angled hexagon \((\vec{L}_i)_{i=1}^6\) in
\(\mathcal{H}^{\theta}\). Since the Riemannian angle of \(\vec{L}_i\) and
\(\vec{L}_{i+1}\) is fixed at \(\pi/2\) by the definition of a right-angled
hexagon, their angle of holomorphy is also \(\theta\); see
Subsection~\ref{subsec:moduli}.

\subsection{Oriented right-angled hexagons}

We will denote by \(L_i\) a real geodesic, by \(\vec{L}_i\) an oriented real
geodesic, and by \(L^{\mathbb{C}}_i\) the complex line containing \(L_i\).

\begin{defn}\label{defn:rahexagon}
Let \((\vec{L}_i)_{i=1}^6\) be six cyclically ordered, oriented real geodesics in
\(\mathbf H^2_{\mathbb C}\). We say that \((\vec{L}_i)_{i=1}^6\) forms an
\emph{oriented right-angled hexagon} if \(\vec{L}_i\) and \(\vec{L}_{i+1}\), with indices
taken modulo \(6\), meet at a Riemannian right angle.
\end{defn}

The orientation is assigned to each complete line \(\vec{L}_i\), independently
of the cyclic order. Put \(x_i=\vec{L}_i\cap \vec{L}_{i+1}\), with indices taken
modulo six, and assume \(x_{i-1}\ne x_i\). The side segment joins \(x_{i-1}\)
to \(x_i\). Its signed length \(\ell_i\) is positive when the segment is
traversed from \(x_{i-1}\) to \(x_i\) in the chosen orientation of
\(\vec{L}_i\), and negative otherwise. At either
endpoint the oriented tangent points in the positive direction of the
complete line; in particular, it is not automatically negated at the terminal vertex.

The side map used in Lemma \ref{thm:closure}  sends \(x_{i-1}\) to \(x_i\), preserves \(\vec{L}_i\)
as an oriented line, and sends the oriented \(\vec{L}_{i-1}\) to the oriented
\(\vec{L}_{i+1}\). Preserving a line means preserving it setwise, not fixing
each of its points. Reversing an orientation therefore requires the
corresponding side parameters to be reconsidered.
These conventions follow \cite[Appendix A]{TanWongZhang}.

\begin{defn}\label{defn:complexrahexagon}
An oriented right-angled hexagon is \emph{complex right-angled}
if the complex geodesics containing consecutive sides are orthogonal.
Equivalently, its complex angle at every vertex is \(\pi/2\).
\end{defn}

This is the case occurring in the Mostow construction
\cite[Section~2.3]{DerauxFalbelPaupert}.
The moduli of such configurations were studied by Cunha,
Dutenhefner, Gusevskii, and Thebaldi~\cite{Cunha2012}.

\subsection{Moduli space of right-angled hexagons}\label{subsec:moduli}

\begin{defn}\label{defn:nondegenerate}
An oriented right-angled hexagon \((\vec L_i)_{i=1}^6\) is called
\emph{nondegenerate} if
\begin{enumerate}
\item its six vertices \(x_1,\ldots,x_6\) are pairwise distinct;
\item it is not contained in a single complex geodesic;
\item it is not contained in a totally real, totally geodesic plane.
\end{enumerate}
\end{defn}

Let \(\mathcal{H}\) denote the space of oriented right-angled hexagons up to
\(\Pu\), with the side labels fixed. Unless otherwise stated, we restrict to
the nondegenerate locus of Definition~\ref{defn:nondegenerate}. No common
complex angle condition is imposed in the definition of \(\mathcal{H}\).

\begin{prop}\label{prop:dimH}
The nondegenerate locus of \(\mathcal H\) has real dimension \(10\).
\end{prop}

\begin{proof}
An oriented real geodesic is determined by an ordered pair of
distinct points of \(\partial\HdC\cong S^3\); hence the space of
such geodesics has real dimension \(6\).
Choose the three alternating sides \(\vec{L}_1, \vec{L}_3, \vec{L}_5\).
Each adjacent pair among them is joined by one of the remaining sides,
perpendicular at both ends. In strictly negative curvature, two complete
geodesics joined in this way are ultraparallel: they are disjoint
and have no common ideal endpoint. Their common perpendicular
is unique.

Conversely, start with three pairwise ultraparallel real geodesics.
Their pairwise common perpendiculars determine real geodesics
\(\vec{L}_2, \vec{L}_4, \vec{L}_6\).
The requirement that the two perpendicular feet on each odd side
be distinct is open. On this locus the construction gives a
right-angled hexagon, and it varies smoothly with the three
alternating sides. Orienting the three perpendicular lines involves
only finitely many choices. The space before quotienting by \(\Pu\)
therefore has real dimension \(3\cdot6=18\).

If all consecutive complex spans coincide, the hexagon is contained
in one complex geodesic. Otherwise, some adjacent pair has distinct
complex spans. An isometry stabilizing the labeled oriented hexagon
preserves this pair and is the identity by
Lemma~\ref{lem:rigidity} below. The vertex assignment to \(x_1\) is
equivariant, and \(\Pu\) acts properly on \(\HdC\), so its
action on this configuration space is proper as well.
The quotient on this locus consequently has real dimension
\[
18-\dim_{\mathbb R}\Pu=18-8=10.
\]
\end{proof}

The determining data in this argument are three \emph{alternating}
sides, not three consecutive sides. The other sides are their common
perpendiculars. Configurations in a single complex geodesic have an
additional normal stabilizer and are considered separately below.

Let \(\mathcal H_{\mathbb C}\subset\mathcal H\) denote the space of complex
right-angled hexagons. To compare with~\cite{Cunha2012}, take six distinct
points in \(\HdC\) in general position, meaning that no three lie in one
complex geodesic, and assume in addition that the complex geodesics
determined by consecutive points are orthogonal. Then these complex side
geodesics \(L^{\mathbb{C}}_1,\ldots,L^{\mathbb{C}}_6\) form a \emph{proper
right-angled \(c\)-hexagon} in the terminology of that paper; its vertices
are interior rather than ideal.

\begin{prop}[{\cite[Theorem~4.6]{Cunha2012}}]\label{prop:dimHC}
On this proper locus, \(\mathcal H_{\mathbb C}\) has real dimension \(4\).
Assigning orientations to the real sides does not change this dimension.
\end{prop}

\begin{proof}
We explain the correspondence with the cited moduli problem.
The three alternating complex geodesics
\(L^{\mathbb{C}}_1, L^{\mathbb{C}}_3, L^{\mathbb{C}}_5\) are
pairwise ultraparallel. Each pair has a unique complex geodesic
orthogonal to both; these are
\(L^{\mathbb{C}}_2,L^{\mathbb{C}}_4,L^{\mathbb{C}}_6\).
The intersections
\(x_i=L^{\mathbb{C}}_i\cap L^{\mathbb{C}}_{i+1}\)
recover the vertices. The real line
\(L_i\subset L^{\mathbb{C}}_i\) is then the unique real geodesic
through \(x_{i-1}\) and \(x_i\). Thus passing from the complex
configuration to our real sides adds only their orientation choices.

For the dimension count, a complex geodesic is determined by its
positive polar point in \(\mathbb{CP}^2\), so the space of complex
geodesics has real dimension \(4\). Triples of complex geodesics
satisfying the stated strict disjointness and nondegeneracy conditions
form an open set of real dimension \(12\).
Their polar vectors are linearly independent: otherwise the
three alternating complex geodesics would have the same complex
common perpendicular, collapsing consecutive vertices.
Normalize the three polar vectors to have norm \(1\).
Their pairwise Hermitian products are nonzero. A unitary map
preserving each polar line must therefore multiply all three
vectors by the same unit scalar. Since they form a basis, it is
projectively the identity. The stabilizer in \(\Pu\) is trivial,
and the quotient has real dimension
\[
3\cdot4-8=4.
\]
The reconstruction above identifies it with the unoriented
configuration space of~\cite[Theorem~4.6]{Cunha2012}.
The finitely many choices of real side orientations add no
continuous parameters.
\end{proof}

For \(\theta\in[0,\pi/2]\), let \(\mathcal H^\theta\) be the subspace
of \(\mathcal H\) consisting of nondegenerate hexagons for which all
six unsigned complex angles equal \(\theta\). By definition,
\[
\mathcal H^{\pi/2}=\mathcal H_{\mathbb C}.
\]
When \(\theta=0\), consecutive real sides have the same complex
span. Propagating this equality around the hexagon shows that all
six sides lie in a single complex geodesic. Intrinsically this is
a complex hyperbolic line of curvature \(-1\).
The ordinary nondegenerate right-angled hexagons there have three
real moduli, given by three alternating positive side lengths;
labeling and orienting the sides only add discrete choices.
Thus \(\mathcal H^0\) has real dimension \(3\).
Hexagons in a \emph{totally real} plane also have three real
moduli, but lie in \(\mathcal H^{\pi/2}\), and the restricted
curvature is \(-1/4\).

These subspaces are defined using unsigned angles. For
\(0<\theta<\pi/2\), Theorem~\ref{thm:main} additionally requires
the signed orientation condition \eqref{eq:signed-angle};
membership in \(\mathcal H^\theta\) alone does not supply it.

\subsection{Side parameters of right-angled hexagons}\label{subsection:sideparameters}

For the hexagons in Theorem~\ref{thm:main}, let \(\tau_j\) be the
unique holomorphic isometry preserving oriented \(\vec{L}_j\) and
sending oriented \(\vec{L}_{j-1}\) to oriented \(\vec{L}_{j+1}\).
Lemma~\ref{thm:closure} later proves that these maps exist under
the condition~\eqref{eq:signed-angle}.
The map \(\tau_j\) has signed translation length \(\ell_j\), as
defined in Subsection \ref{subsection:eigenvalues}, and normal rotation angle \(\phi_j\) modulo \(2\pi\).
Its \emph{complex length parameter} is
\[
\lambda_j=e^{\ell_j/2-i\phi_j/3}.
\]
It is defined up to a cube root of unity, corresponding to the
choice of lift of \(\tau_j\) to \(\Su\).
In particular, \(|\lambda_j|\) may be less than \(1\) when
\(\ell_j<0\).

\subsection{Orthogonality and rigidity}

\begin{prop}
For two intersecting real geodesics \(L_1\) and \(L_2\) in the complex hyperbolic plane, if the
complex geodesics containing them are orthogonal, the real geodesics
\(L_1\) and \(L_2\) meet at a Riemannian right angle. The converse
does not hold in general.
	
\end{prop}

\begin{proof}
Complex orthogonality means that the Hermitian product of their
tangents vanishes. Its real part therefore vanishes as well.
For the converse, the tangent vectors \((1,0)\) and \((i,0)\)
at the origin of the ball are Riemannian orthogonal but span the
same complex line.
\end{proof}

For the closure argument in Lemma \ref{thm:closure} we also need the following rigidity statement.

\begin{lem}\label{lem:rigidity}
	Let \(g\in\Pu\) preserve two intersecting oriented real geodesics
	whose complex spans are distinct. Then \(g\) is the identity. Any
	lift of \(g\) to \(\Su\) is a scalar cube root of the identity.
\end{lem}

\begin{proof}
The unique intersection point is fixed by $g$. Move it to the origin of
the ball model, where its isotropy group acts on the tangent space
as \(\mathrm U(2)\). Since the geodesics are oriented, their positive
unit tangents are fixed individually. These two vectors are
complex-linearly independent, so the unitary transformation is the
identity.
\end{proof}

In the complex-orthogonal case, the same conclusion can be seen
directly in the original ball coordinates. Take the oriented axes
\[
\vec{L}_1=\{(x,0):-1<x<1\},\qquad
\vec{L}_2=\{(0,x):-1<x<1\},
\]
with their positive coordinate orientations. The polar vectors of
their complex spans are
\begin{equation}\label{eq:polar-vector}
\mathbf n_1=\begin{pmatrix}0\\1\\0\end{pmatrix},\qquad
\mathbf n_2=\begin{pmatrix}1\\0\\0\end{pmatrix}.
\end{equation}
A lift of \(g\) in  \(\Su\) preserving both lines fixes their
intersection projectively and preserves both polar lines,
 so it has the form
\[
\operatorname{diag}\bigl(e^{i\theta_1},e^{i\theta_2},
e^{-i(\theta_1+\theta_2)}\bigr).
\]
Its affine action on the two coordinate axes multiplies them by
\(e^{i(2\theta_1+\theta_2)}\) and
\(e^{i(\theta_1+2\theta_2)}\), respectively.
Preserving both real orientations requires
\[
2\theta_1+\theta_2\equiv0,\qquad
\theta_1+2\theta_2\equiv0\pmod{2\pi}.
\]
Hence \(\theta_1\equiv\theta_2\) and
\(3\theta_1\equiv0\pmod{2\pi}\), giving a scalar cube root
of the identity. This coordinate computation uses complex
orthogonality; the preceding tangent-space proof covers all
pairs with distinct complex spans.

The distinct-span hypothesis excludes \(\theta=0\). In that case
both lines lie in one complex geodesic, and rotation in its complex
normal direction fixes both lines pointwise.

\section{The Delambre--Gauss formulas}
\label{sec:transport}


In this section, we prove Theorem~\ref{thm:main}.
We first find an isometry preserving each oriented side and
carrying the preceding oriented side to the following one.
The closure argument tracks two initial lines under these six maps.
Equations~\eqref{eq:tau} and~\eqref{eq:eta} are the main matrix identities we need for the proof of Theorem~\ref{thm:main}.
We use the notation and calculations from \cite[Lemma A.8]{TanWongZhang}.

\begin{lem}[Side maps and closure]\label{thm:closure}

For a hexagon \((\vec{L}_n)_{n=1}^6\) satisfying the hypotheses
of Theorem~\ref{thm:main}, there are unique side maps
\[
\tau_n \in \Pu,\qquad n=1,\ldots,6,
\]
satisfying
\[
\tau_n(\vec{L}_n)=\vec{L}_n
\quad\text{and}\quad
\tau_n(\vec{L}_{n-1})=\vec{L}_{n+1},
\]
with indices modulo \(6\). Set
\[
\eta_n := (\tau_{n-1}\cdots \tau_1)^{-1}\tau_n(\tau_{n-1}\cdots \tau_1).
\]
Then
\begin{equation}
\label{eq:tau}
\tau_6\tau_5\tau_4\tau_3\tau_2\tau_1=\mathrm{id},
\end{equation}
and
\begin{equation}
\label{eq:eta}
\eta_1\eta_2\eta_3\eta_4\eta_5\eta_6=\mathrm{id}.
\end{equation}
Furthermore,
\[
\begin{aligned}
\eta_n(\vec{L}_1)&=\vec{L}_1 &&\text{for } n=1,3,5,\\
\eta_n(\vec{L}_6)&=\vec{L}_6 &&\text{for } n=2,4,6.
\end{aligned}
\]

\end{lem}


\begin{proof}
First consider the side \(L_n\). Translate along it from \(x_{n-1}\)
to \(x_n\), preserving its positive tangent. At these endpoints,
\eqref{eq:signed-angle} gives the same Hermitian product of this
tangent with the positive tangents to \(L_{n-1}\) and \(L_{n+1}\):
\[
\langle u_n,u_{n-1}\rangle_{x_{n-1}}
=\langle u_n,u_{n+1}\rangle_{x_n}.
\]
Consequently, after translation, these transverse tangents have
the same component along the complex tangent to \(L_n\).
Their remaining components lie in its one-dimensional complex
orthogonal complement and both have norm \(\sin\theta>0\).
A unique unitary rotation of this complement takes one to the other.
Composing this rotation with the translation gives \(\tau_n\);
uniqueness of  \(\tau_n\) also follows from Lemma~\ref{lem:rigidity}.

We now prove that \(\tau_6\cdots \tau_2\tau_1=\mathrm{id}\). Set
\[
\tau=\tau_6\tau_5\tau_4\tau_3\tau_2\tau_1.
\]
A direct check shows that
\[
\tau(\vec{L}_6)=\vec{L}_6
\quad\text{and}\quad
\tau(\vec{L}_1)=\vec{L}_1.
\]
Indeed, the action of \(\tau\) on these two lines can be read off from the following chain:
\[
\begin{aligned}
\vec{L}_1
&\xrightarrow{\tau_1} \vec{L}_1
\xrightarrow{\tau_2} \vec{L}_3
\xrightarrow{\tau_3} \vec{L}_3
\xrightarrow{\tau_4} \vec{L}_5
\xrightarrow{\tau_5} \vec{L}_5
\xrightarrow{\tau_6} \vec{L}_1,\\
\vec{L}_6
&\xrightarrow{\tau_1} \vec{L}_2
\xrightarrow{\tau_2} \vec{L}_2
\xrightarrow{\tau_3} \vec{L}_4
\xrightarrow{\tau_4} \vec{L}_4
\xrightarrow{\tau_5} \vec{L}_6
\xrightarrow{\tau_6} \vec{L}_6.
\end{aligned}
\]
The two initial lines have distinct complex spans because \(\theta>0\).
Lemma~\ref{lem:rigidity} therefore gives \(\tau=\mathrm{id}\).

To make the order of multiplication explicit, the first three
conjugated factors satisfy
\[
\begin{aligned}
\eta_1\eta_2
 &=\tau_1(\tau_1^{-1}\tau_2\tau_1)=\tau_2\tau_1,\\
\eta_1\eta_2\eta_3
 &=\tau_2\tau_1(\tau_2\tau_1)^{-1}
   \tau_3(\tau_2\tau_1)=\tau_3\tau_2\tau_1.
\end{aligned}
\]
Inductively, \(\eta_1\cdots\eta_n=\tau_n\cdots\tau_1\).
At \(n=6\) this gives the asserted closure.

For odd \(n\), the same tracking chains show that
\(\tau_{n-1}\cdots\tau_1\) sends \(\vec L_1\) to \(\vec L_n\).
Since \(\tau_n\) preserves \(\vec L_n\), we obtain
\[
\begin{aligned}
\eta_n(\vec L_1)
 &=(\tau_{n-1}\cdots\tau_1)^{-1}
   \tau_n(\tau_{n-1}\cdots\tau_1)(\vec L_1)\\
 &=(\tau_{n-1}\cdots\tau_1)^{-1}(\vec L_n)=\vec L_1.
\end{aligned}
\]
For even \(n\), the product \(\tau_{n-1}\cdots\tau_1\)
instead sends \(\vec L_6\) to \(\vec L_n\), and the same calculation
gives
\[
\eta_n(\vec L_6)
=(\tau_{n-1}\cdots\tau_1)^{-1}(\vec L_n)=\vec L_6.
\]
All these statements concern the oriented complete lines.
\end{proof}

\begin{proof}[Proof of Theorem~\ref{thm:main}]

We use the Siegel domain model of the complex hyperbolic plane.
The alternating stabilizers in Lemma~\ref{thm:closure} put \(\vec L_1\) and \(\vec L_6\)
into normal positions. Set
\[
\vec L_1=\vec L_{[0,\infty]},
\]
where \(\vec L_{[0,\infty]}\) is the oriented real geodesic from $0$ to $\infty$ in \(\mathbf H^{2}_{\mathbb C}\).
The second reference line is \(\vec L_6=Q(\vec L_1)\), with \(Q \in \Pu\)
given below. The signed angle condition ensures that this normalization
preserves both orientations.
For \(n=1,3,5\), the isometries \(\eta_n\in\mathrm{PU}(2,1)\)
are represented by
\[
A_n=
\begin{pmatrix}
\lambda_n&0&0\\
0&\overline{\lambda_n}\,\lambda_n^{-1}&0\\
0&0&\overline{\lambda_n}^{-1}
\end{pmatrix},
\qquad
\lambda_n=e^{\ell_n/2-i\phi_n/3},
\]
where \(\lambda_n\) is as in Subsection~\ref{subsection:sideparameters},
and comes from the map \(\tau_n\) of Lemma~\ref{thm:closure}.

These matrices fix \(0\) and \(\infty\). Let \(\mathbf p_1\) and \(\mathbf q_1\)
be the lifts of \(0\) and \(\infty\) given by
\[
\mathbf p_1=
\begin{pmatrix}0\\0\\-1\end{pmatrix},
\qquad
\mathbf q_1=
\begin{pmatrix}1\\0\\0\end{pmatrix},
\]
so that \(\langle {\mathbf p}_1,\mathbf q_1\rangle=-1\). The geodesic $\vec L_1$ with
endpoints \(0\) and \(\infty\) is parametrized by 
\[
\gamma_1(s)=e^{s/2}\mathbf q_1+e^{-s/2}\mathbf p_1.
\]

For the remaining sides, let
\begin{equation}\label{eq:Q}
Q=
\begin{pmatrix}
\dfrac{1-i\cos\theta}{2}&\dfrac{\sin\theta}{\sqrt2}&\dfrac{-1-i\cos\theta}{2}\\[4mm]
\dfrac{\sin\theta}{\sqrt2}&-i\cos\theta&\dfrac{\sin\theta}{\sqrt2}\\[4mm]
\dfrac{-1-i\cos\theta}{2}&\dfrac{\sin\theta}{\sqrt2}&\dfrac{1-i\cos\theta}{2}
\end{pmatrix}.
\end{equation}
One checks that \(Q^*H_SQ=H_S\), \(Q^{-1}=Q^*\), and
\(\det Q=-1\). Thus \(Q\) represents a projective unitary isometry,
and conjugation by \(Q\) preserves \(\Su\).
For \(n=2,4,6\), define
\[
A_n=
Q
\begin{pmatrix}
\lambda_n&0&0\\
0&\overline{\lambda_n}\lambda_n^{-1}&0\\
0&0&\overline{\lambda_n}^{-1}
\end{pmatrix}
Q^{-1},
\]
where \(\lambda_n\) is as in Subsection~\ref{subsection:sideparameters},
and comes from the map \(\tau_n\) of Lemma~\ref{thm:closure}.
Then \(A_n\) fixes \(Q(0)\) and \(Q(\infty)\). Let
\(\mathbf p_2=Q(\mathbf p_1)\) and \(\mathbf q_2=Q(\mathbf q_1)\) be the lifts
of \(Q(0)\) and \(Q(\infty)\), so that
\[
\mathbf p_2=
\begin{pmatrix}
\dfrac{1+i\cos\theta}{2}\\[4mm]
-\dfrac{\sin\theta}{\sqrt2}\\[4mm]
\dfrac{-1+i\cos\theta}{2}
\end{pmatrix},
\qquad
\mathbf q_2=
\begin{pmatrix}
\dfrac{1-i\cos\theta}{2}\\[4mm]
\dfrac{\sin\theta}{\sqrt2}\\[4mm]
\dfrac{-1-i\cos\theta}{2}
\end{pmatrix},
\]
with \(\langle {\mathbf p}_2, {\mathbf q}_2\rangle=-1\). The geodesic $\vec L_6$ with
endpoints \(Q(0)\) and \(Q(\infty)\) is parametrized by
\[ 
\gamma_2(s)=e^{s/2}\mathbf q_2+e^{-s/2}\mathbf p_2.
\]

The point
\[
\mathbf p_0=
\begin{pmatrix}1\\0\\-1\end{pmatrix}
\]
is the intersection point of \(\gamma_1\) and \(\gamma_2\) at \(s=0\). Indeed,
\(\gamma_1(0)=\mathbf p_1+\mathbf q_1=\mathbf p_0\) and
\(\gamma_2(0)=\mathbf p_2+\mathbf q_2=\mathbf p_0\).
The complex spans of the geodesics \(\gamma_1\) and \(\gamma_2\) have
positive unit polar vectors \(\mathbf c_1\) and \(\mathbf c_2\), respectively, where
\[
\mathbf c_1=\begin{pmatrix}0\\1\\0\end{pmatrix},\qquad
\mathbf c_2=Q \mathbf c_1=
\begin{pmatrix}
\sin\theta/\sqrt2\\ -i\cos\theta\\ \sin\theta/\sqrt2
\end{pmatrix}.
\]
Indeed, \(\mathbf c_i\) is orthogonal to both \(\mathbf p_i\) and
\(\mathbf q_i\), and \(\langle \mathbf c_i, \mathbf c_i\rangle=1\).
Since \(|\langle \mathbf c_1, \mathbf c_2\rangle|=\cos\theta\), the complex
angle of these two complex geodesics is \(\theta\).

For the Riemannian angle of $\{\gamma_1,\gamma_2\}$  and its orientation sign, apply the Cayley
transform \(C\)  from the Siegel model to the ball model in Section~\ref{sec:models}.
The lift \(\widetilde{\mathbf p}_0=\mathbf p_0/\sqrt2\) has norm
\(-1\), and
\[
C\mathbf p_0=\begin{pmatrix}0\\0\\\sqrt2\end{pmatrix},\quad
C\mathbf q_1=\frac1{\sqrt2}\begin{pmatrix}1\\0\\1\end{pmatrix},\quad
C\mathbf q_2=\frac1{\sqrt2}
\begin{pmatrix}-i\cos\theta\\ \sin\theta\\1\end{pmatrix}.
\]
Thus \(p_0\) goes to the origin of the ball.
The positive directions of \(\gamma_1\) and \(\gamma_2\), toward
\(q_1\) and \(q_2\) respectively, are represented by
\[
v_1=\begin{pmatrix}1\\0\end{pmatrix},\qquad
v_2=\begin{pmatrix}-i\cos\theta\\ \sin\theta\end{pmatrix}.
\]
In a unitary tangent frame, \(\langle v_1,v_2\rangle=i\cos\theta\).
Its real part is zero and its modulus is \(\cos\theta\).
Moreover, \(\langle v_2,v_1\rangle=-i\cos\theta\), as required
for the ordered pair \((\vec L_6, \vec L_1)\) in \eqref{eq:signed-angle}.

The closure relation in Lemma~\ref{thm:closure} holds in \(\Pu\).
Its lifts therefore satisfy
\begin{equation}
A_1A_2A_3A_4A_5A_6=\varepsilon I,\qquad \varepsilon^3=1.
\label{eq:product-epsilon}
\end{equation}
For each \(n\in\{1,2,\dots,6\}\), \(A_n\) is a matrix depending on
the parameters \(\theta\) and \(\lambda_n\), so we write
\(A_n=A_n(\theta,\lambda_n)\).
Since \(A_n(\theta,\lambda)^{-1}=A_n(\theta,\lambda^{-1})\),
moving the last three factors of \eqref{eq:product-epsilon} to the right gives
\eqref{eq:transport-main}, with inverse factors in the order \(6,5,4\). This ends the proof of Theorem~\ref{thm:main}
\end{proof}

\begin{rem}\label{rem:scalar-dg}
Equating the entries of \eqref{eq:transport-main} gives nine
scalar Delambre--Gauss identities. They need not be independent:
the equation expresses closure in the eight-real-dimensional
group \(\Pu\), with a central lift in \(\Su\).
The matrix form retains the multiplication order and displays
the dependence on \(\theta\) without reproducing lengthy
entry expansions.
\end{rem}

\section{Right-angled hexagons in the real hyperbolic plane and the complex hyperbolic line}
\label{sec:real-hyperbolic-degeneration}

In this section, we distinguish two geometrically different cases,
namely the two degenerate cases of Theorem~\ref{thm:main}.
If the complex angle \(\theta\) is \(\pi/2\), then the right-angled
hexagon \((\vec{L}_i)_{i=1}^6\) is contained in a totally real plane.
If the complex angle \(\theta\) is \(0\), then consecutive sides lie in
the same complex geodesic, so the entire hexagon lies in that
complex geodesic. The latter case is excluded from
Theorem~\ref{thm:main}.

\subsection{Right-angled hexagons in the real hyperbolic plane}

A totally real, totally geodesic plane \(\mathcal P\) in \(\mathbf H^{2}_{\mathbb C}\) has curvature \(-1/4\)
for the restricted Bergman metric. If \(d_{\mathcal P}\) denotes
this restricted distance and \(d_{-1}\) the distance after
normalizing its curvature to \(-1\), then
\[
d_{\mathcal P}=2d_{-1}.
\]
This factor must be taken into account when comparing side lengths
in the two matrix descriptions of real hyperbolic geometry: one
given by \(\mathrm{SU}(1,1) < \mathrm{SU}(2,1)\) and the other by \(\mathrm{SO}_0(2,1) < \mathrm{SU}(2,1)\). 
Recall that the real hyperbolic plane \(\mathbf H^{2}_{\mathbb R}\) has the Poincar\'e disc model and the Klein disc model.
 Let $\Phi$ be the map from the Poincar\'e disc to the Klein disc given by
\[
\Phi: z=x+iy\mapsto \left(\frac{2x}{1+x^2+y^2},  \frac{2y}{1+x^2+y^2}\right).
\]
Here both disc models have curvature \(-1\); \(\Phi\) does not include
the rescaling to the restricted Bergman metric.
Passing between the two models gives the following homomorphism;
see~\cite[Section~4.1 and Exercise~4.2.1]{Par2008}.

\begin{prop} \label{prop:realhypplane1}
There is a homomorphism $\rho : \mathrm{SU}(1,1) \longrightarrow \mathrm{SO}_0(2,1)$, with kernel $\{\pm I\}$, given on generators, for real $\lambda$ and $\theta$, by
\[
\rho\!\begin{pmatrix}
\cosh \lambda & \sinh \lambda \\[2pt]
\sinh \lambda & \cosh \lambda
\end{pmatrix}
=
\begin{pmatrix}
\cosh 2\lambda & 0 & \sinh 2\lambda \\[2pt]
0 & 1 & 0 \\[2pt]
\sinh 2\lambda & 0 & \cosh 2\lambda
\end{pmatrix},
\]
\[
\rho\!\begin{pmatrix}
e^{i\theta} & 0 \\[2pt]
0 & e^{-i\theta}
\end{pmatrix}
=
\begin{pmatrix}
\cos 2\theta & -\sin 2\theta & 0 \\[2pt]
\sin 2\theta & \cos 2\theta & 0 \\[2pt]
0 & 0 & 1
\end{pmatrix}.
\]
\end{prop}

\begin{prop}\label{prop:realhypplane2}
There is a homomorphism $\rho_1 : \mathrm{SL}(2,\mathbb{R}) \longrightarrow \mathrm{SU}(1,1)$ given by
\[
\rho_1\!\begin{pmatrix}
a & b \\
c & d
\end{pmatrix}
=
\begin{pmatrix}
\alpha & \beta \\
\overline{\beta} & \overline{\alpha}
\end{pmatrix},
\]
where
\[
\alpha = \frac{1}{2}\bigl(a + d - i(b - c)\bigr),
\qquad
\beta = \frac{1}{2}\bigl(a - d + i(b + c)\bigr),
\]
with $a,b,c,d \in \mathbb{R}$.
\end{prop}

The map \(\rho_1\) is conjugation by
\[
\frac{1}{\sqrt2}\begin{pmatrix}1&i\\1&-i\end{pmatrix},
\]
so it is a homomorphism. The composition \(\rho\circ\rho_1\)
relates the \(\mathrm{SL}(2,\mathbb R)\) and
\(\mathrm{SO}_0(2,1)\) descriptions of orientation-preserving real hyperbolic
isometries. Applying it to a \(2\times2\) closure identity sends
the central sign \(\pm I\) to the identity.

\begin{cor}[Totally real case]\label{cor:real-case}
If a right-angled hexagon in Theorem~\ref{thm:main} lies in a totally real,
totally geodesic plane \(\mathcal P\), then \(\theta=\pi/2\).
Its side maps in Lemma \ref{thm:closure}  preserve \(\mathcal P\), and their projective
\(3\times3\) closure identity is the real hyperbolic closure
identity in the projective Lorentz group $\mathrm{SO}_0(2,1)$, with
\(d_{\mathcal P}=2d_{-1}\).
\end{cor}

\begin{proof}
The tangent directions to \(\mathcal P\) at a right-angled
vertex are Hermitian orthogonal. This gives \(\theta=\pi/2\).
Two intersecting real geodesics in \(\mathcal P\) determine
that totally real plane. Since \(\tau_j\) preserves \(\vec{L}_j\)
and sends \(\vec{L}_{j-1}\) to \(\vec{L}_{j+1}\), it preserves \(\mathcal P\).
In ball coordinates of the complex hyperbolic plane in which \(\mathcal P\) is real, these
projective actions are represented by real Lorentz matrices.
Their product is the identity by Lemma~\ref{thm:closure},
giving the stated real hyperbolic interpretation.
\end{proof}

The maps \(\rho\) and
\(\rho_1\) in Propositions \ref{prop:realhypplane1} and \ref{prop:realhypplane2} describe the component preserving the orientation of
the real plane. Side maps of a right-angled hexagon reversing that orientation lie outside
this component and cannot be identified with images under \(\rho\).
The condition
\(\theta=\pi/2\) alone does not imply that all six sides lie in a
single totally real plane; see the Mostow example in
Section~\ref{sec:mostow-canonical-hexagon}.

\subsection{Right-angled hexagons in  the complex hyperbolic line}

In the ball models of the complex hyperbolic line and the complex hyperbolic plane,
the embedding of the isometry group of the complex hyperbolic line into that of
the complex hyperbolic plane which preserves the complex geodesic \(z_2=0\) is
\[
\rho_2:\mathrm{SU}(1,1)\longrightarrow\mathrm{SU}(2,1),\qquad
\begin{pmatrix}a&b\\\bar b&\bar a\end{pmatrix}
\longmapsto
\begin{pmatrix}a&0&b\\0&1&0\\\bar b&0&\bar a\end{pmatrix}.
\]
The \(2\times2\) block acts on the first and third coordinates, one
positive and one negative for \(H_B=\operatorname{diag}(1,1,-1)\).
Unlike the embedding \(\rho\) in Proposition~\ref{prop:realhypplane1},
this embedding does not send the central element of \(\mathrm{SU}(1,1)\)
to a scalar matrix in the ambient group:
\[
\rho_2(-I_2)=\operatorname{diag}(-1,1,-1),
\] where \(I_2\) is the \(2 \times 2\) identity matrix.
This matrix acts trivially on the complex geodesic but is not central
in \(\mathrm{SU}(2,1)\). Thus a \(2\times2\) closure identity with
right-hand side \(-I_2\) does not become the ambient central closure
identity under \(\rho_2\).

There are also geometric restrictions at \(\theta=0\).
The prescribed side maps may reverse the orientation of the complex
geodesic and hence fail to be holomorphic there. Even when such
holomorphic maps exist, their rotations in the ambient complex normal
direction are not determined by the side lines. The embedding above
therefore compares the matrix models but does not extend
Theorem~\ref{thm:main} to \(\theta=0\).

\section{An example: Mostow's canonical hexagon}
\label{sec:mostow-canonical-hexagon}

We test the Delambre--Gauss identity in
	Theorem \ref{thm:main} on the canonical right-angled  hexagon associated
with Mostow's groups~\cite{Mostow}, using the construction of
Deraux, Falbel, and Paupert~\cite[Sections~2.3--2.4]{DerauxFalbelPaupert}.

Let \(p\in\{3,4,5\}\). Put
\[
\alpha=\frac{1}{2\sin(\pi/p)},\qquad
\eta=e^{\pi i/p},\qquad \phi=e^{\pi it/3},
\]
and use the Hermitian matrix \(H\), and the 
\(\mathbb{Z}_3\)-symmetry given by \(J\):
\[
H=\begin{pmatrix}
1&-\alpha\bar\phi&-\alpha\phi\\
-\alpha\phi&1&-\alpha\bar\phi\\
-\alpha\bar\phi&-\alpha\phi&1
\end{pmatrix},\qquad
J=\begin{pmatrix}0&0&1\\1&0&0\\0&1&0\end{pmatrix}.
\]
This is the matrix for the convention
\(\langle z,w\rangle=w^*Hz\) used throughout this section.
Although \(H\) has signature \((2,1)\) for
\(|t|<3(1/2-1/p)\), the six vertices used here are interior points
only in the smaller range
\begin{equation}\label{eq:mostow-range}
|t|<\frac12-\frac1p.
\end{equation}
We impose \eqref{eq:mostow-range} throughout this section.

The three complex reflections \(R_1,R_2,R_3\) of order \(p\)
in the Mostow configuration satisfy
\(R_iR_jR_i=R_jR_iR_j\).
Write \(p_{12}\) for the intersection of the mirrors of \(R_1,R_2\),
and set \(p_{23}=Jp_{12}\), \(p_{31}=J^2p_{12}\).
We use the lifts
\[
\mathbf p_{12}=
\begin{pmatrix}
\alpha\phi+\alpha^2\bar\phi^2\\
\alpha\bar\phi+\alpha^2\phi^2\\
1-\alpha^2
\end{pmatrix},\qquad
\mathbf p_{23}=J {\mathbf p}_{12},\qquad {\mathbf p}_{31}=J^2 {\mathbf p}_{12}.
\]
Their norms agree because \(J\) is unitary with respect to the matrix $H$.
The bisector \(B_1\) is equidistant from \(p_{12}\) and \(p_{31}\);
put \(B_2=J(B_1)\), \(B_3=J^2(B_1)\).
Their common intersection
\[
S_\eta=B_1\cap B_2
\]
is a smooth, non-totally-geodesic Giraud disk.
Proposition~2.12 of~\cite{DerauxFalbelPaupert} identifies six
real geodesic segments in $S_\eta$ forming a right-angled hexagon.
The region bounded by these segments is the hexagonal \(2\)-face
used in that construction.

\subsection{Recovering the vertices from polar vectors}

We recall the calculation of~\cite[Section~2.4]{DerauxFalbelPaupert}.
The extension of \(S_\eta\) to projective space is
\[
\widetilde S_\eta
=\{[x]\in\mathbb{CP}^2:
 |\langle x, {\mathbf p}_{12}\rangle|
 =|\langle x, {\mathbf p}_{23}\rangle|
 =|\langle x, {\mathbf p}_{31}\rangle|\}.
\]
It is a torus. Let \(\widetilde\sigma_j\) be the projective extension
of the real spine of \(B_j\). The two slices of \(B_j\) that meet
\(S_\eta\) in geodesics have polar points at
\(\widetilde\sigma_j\cap\widetilde S_\eta\).
For \(j=2\), lifts of these points are
\[
\mathbf v_{123}=\begin{pmatrix}-\eta i\bar\phi\\1\\\bar\eta i\phi\end{pmatrix},
\qquad
\mathbf v_{321}=\begin{pmatrix}\bar\eta i\bar\phi\\1\\-\eta i\phi\end{pmatrix}.
\]
The other four vectors are obtained by applying \(J\) and cycling
the indices. In particular, \(\mathbf v_{312}=J^2 {\mathbf v}_{123}\) and
\(\mathbf v_{213}=J^2 {\mathbf v}_{321}\).

The relation \(\alpha=i/(\eta-\bar\eta)\) gives
\[
\begin{aligned}
\mathbf p_{12}+\eta i\phi\, {\mathbf p}_{23}
 &=\alpha^2(\eta^2\phi^2+\bar\eta i\bar\phi)\, {\mathbf v}_{123},\\
\mathbf p_{12}-\bar\eta i\phi\, {\mathbf p}_{23}
 &=\alpha^2(\bar\eta^2\phi^2-\eta i\bar\phi)\, {\mathbf v}_{321}.
\end{aligned}
\]
The coefficients of \(\mathbf p_{23}\) on the left have modulus \(1\).
These formulas place the two projective points on
\(\widetilde\sigma_2\).
Next put
\[
\kappa=\det H=1-3\alpha^2-\alpha^3(\phi^3+\bar\phi^3).
\]
Direct multiplication gives \(H {\mathbf p}_{12}=\kappa(0,0,1)^{\top}\);
cycling coordinates gives the corresponding formulas for
\(H {\mathbf p}_{23}\) and \(H {\mathbf p}_{31}\).
With our convention of linearity in the first variable, it follows that
\begin{align*}
\langle {\mathbf v}_{123}, {\mathbf p}_{12}\rangle&=\bar\eta i\phi\,\kappa, &
\langle {\mathbf v}_{123}, {\mathbf p}_{23}\rangle&=-\eta i\bar\phi\,\kappa, &
\langle {\mathbf v}_{123}, {\mathbf p}_{31}\rangle&=\kappa,\\
\langle {\mathbf v}_{321}, {\mathbf p}_{12}\rangle&=-\eta i\phi\,\kappa, &
\langle {\mathbf v}_{321}, {\mathbf p}_{23}\rangle&=\bar\eta i\bar\phi\,\kappa, &
\langle {\mathbf v}_{321}, {\mathbf p}_{31}\rangle&=\kappa.
\end{align*}
For each of the two functionals, the three values have constant modulus,
so both polar points belong to \(\widetilde S_\eta\). The remaining points
satisfy the same conditions by the symmetry \(J\).

These vectors are polar to the mirrors of the complex reflections
\(J^{\pm1}R_jR_k\) described in~\cite[Section~2.4]{DerauxFalbelPaupert}.
Their relevant orthogonality relations are
\[
\langle {\mathbf v}_{321}, {\mathbf v}_{312}\rangle=0,\qquad
\langle {\mathbf v}_{123}, {\mathbf v}_{213}\rangle=0,
\]
together with their cyclic images.
For \eqref{eq:mostow-range}, all six polar vectors are positive.
Each displayed orthogonal pair therefore spans a positive definite
plane, whose orthogonal complement is a negative line.
The corresponding vertices in \(\HdC\) are
\[
s_{12}=v_{312}^{\perp}\cap v_{321}^{\perp},\qquad
\tilde s_{21}=v_{123}^{\perp}\cap v_{213}^{\perp}.
\]
Here intersections denote their projective points.

Explicitly,
\[
s_{12}=[\bar{\eta}i\phi,\;\bar{\eta}i\bar{\phi},\;1]^{\top},
\qquad
\tilde{s}_{21}=[-\eta i\phi,\;-\eta i\bar{\phi},\;1]^{\top}.
\]
The remaining four vertices are obtained by applying the symmetry \(J\):
\[
s_{23}=J(s_{12}),\qquad
s_{31}=J(s_{23}),\qquad
\tilde{s}_{32}=J(\tilde{s}_{21}),\qquad
\tilde{s}_{13}=J(\tilde{s}_{32}).
\]
Thus the six vertices, in cyclic order, are
\[
s_{12},\; \tilde{s}_{13},\; s_{23},\; \tilde{s}_{21},\; s_{31},\; \tilde{s}_{32}.
\]

We orient the six complete side lines in this cyclic order:
\[
\begin{array}{lll}
\vec{L}_1:s_{12}\longrightarrow\tilde s_{13},&
\vec{L}_2:\tilde s_{13}\longrightarrow s_{23},&
\vec{L}_3:s_{23}\longrightarrow\tilde s_{21},\\
\vec{L}_4:\tilde s_{21}\longrightarrow s_{31},&
\vec{L}_5:s_{31}\longrightarrow\tilde s_{32},&
\vec{L}_6:\tilde s_{32}\longrightarrow s_{12}.
\end{array}
\]
For this choice all six signed side lengths are positive.
The arrows in Figure~\ref{fig:mostow-canonical-hexagon} record these
orientations.

The complex geodesics containing successive sides are orthogonal
\cite[Proposition~2.12]{DerauxFalbelPaupert}. Hence the six
Riemannian angles and the six complex angles all equal \(\pi/2\).
The orientation condition \eqref{eq:signed-angle} is automatic.
Moreover, every pair of consecutive real sides is contained in a
totally real plane that is a meridian of one of the bisectors
\(B_j\), as shown in~\cite[Proposition~2.13]{DerauxFalbelPaupert}.
This is a statement about each adjacent pair; it does not assert that a
single totally real plane contains the whole hexagon.

\begin{rem}
The points \([{\mathbf v}_{ijk}]\) are positive projective points and are
not vertices in complex hyperbolic space. Their polar complex
geodesics determine the negative points \(s_{ij},\tilde s_{ij}\).
Thus \(B_1\cap B_2\cap B_3=S_\eta\) is the disk containing the
hexagon, while the particular vertices are selected by the
projective extension and polarity calculation above.
\end{rem}

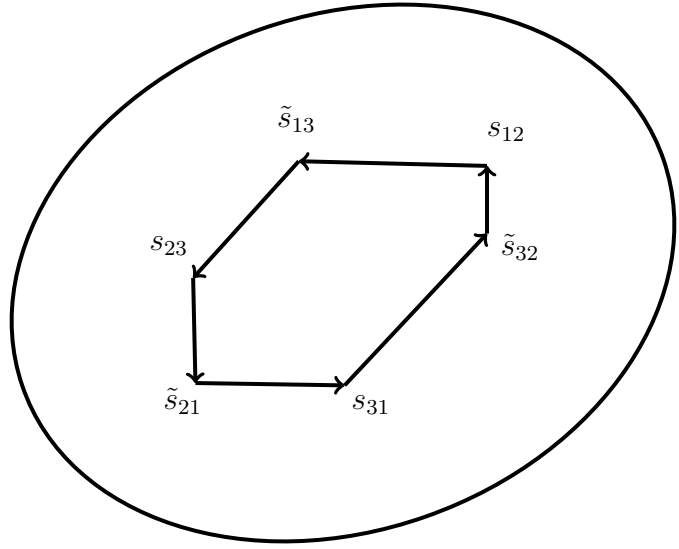
\begin{figure}[htbp]
    \centering
    \tikzset{every picture/.style={line width=0.75pt}}
    \begin{tikzpicture}[x=0.75pt, y=0.75pt, yscale=-1, xscale=1]
    \draw [line width=1.5] (148.74,254.86) .. controls (125.87,187.43) and (179.28,107.5) .. (268.05,76.35)
        .. controls (356.82,45.19) and (447.33,74.6) .. (470.21,142.03)
        .. controls (493.09,209.47) and (439.67,289.39) .. (350.9,320.55)
        .. controls (262.13,351.7) and (171.62,322.3) .. (148.74,254.86) -- cycle;
    \draw [->,line width=1.5] (381.64,144.77) -- (287.31,142.43);
    \draw [->,line width=1.5] (235.54,253.6) -- (310.31,254.77);
    \draw [->,line width=1.5] (234.39,200.94) -- (235.54,253.6);
    \draw [->,line width=1.5] (381.64,178.7) -- (381.64,144.77);
    \draw [->,line width=1.5] (287.31,142.43) -- (234.39,200.94);
    \draw [->,line width=1.5] (310.31,254.77) -- (381.64,178.7);
    \draw (380,121.4) node [anchor=north west][inner sep=0.75pt] {$s_{12}$};
    \draw (211,178.4) node [anchor=north west][inner sep=0.75pt] {$s_{23}$};
    \draw (312.31,258.17) node [anchor=north west][inner sep=0.75pt] {$s_{31}$};
    \draw (275,113.4) node [anchor=north west][inner sep=0.75pt] {$\tilde{s}_{13}$};
    \draw (218,254.4) node [anchor=north west][inner sep=0.75pt] {$\tilde{s}_{21}$};
    \draw (387,177.4) node [anchor=north west][inner sep=0.75pt] {$\tilde{s}_{32}$};
    \end{tikzpicture}
    \caption{Mostow's canonical right-angled hexagon with the chosen side orientations}
    \label{fig:mostow-canonical-hexagon}
\end{figure}

\subsection{The Delambre--Gauss formulas for Mostow's hexagons}
 
For  a Mostow's hexagon the complex angle at every vertex is \(\pi/2\).
We therefore set \(\theta=\pi/2\) in the matrix \(Q\) and verify the matrix relation
\[
A_1A_2A_3=A_6^{-1}A_5^{-1}A_4^{-1},
\]
using the matrices of  Theorem \ref{thm:main}.

For a concrete check, take \(p=4\), \(t=1/10\), with the orientations
specified above. The symmetry \(J\) gives equal parameters on the
three odd sides and on the three even sides.

The distances and normal rotations of the side maps give, for
\(j=2,4,6\),
\[
\ell_j=1.635297750842,\qquad \phi_j=2.905973204571,
\]
and, for \(j=1,3,5\),
\[
\ell_j=3.545892107956,\qquad \phi_j=-2.591813939212.
\]
Hence, using \(\lambda_j=e^{\ell_j/2-i\phi_j/3}\),
\[
\lambda_1=\lambda_3=\lambda_5
 \simeq3.824063576608+4.477403179458i,
\]
and
\[
\lambda_2=\lambda_4=\lambda_6
 \simeq1.283005214949-1.866784171584i.
\]

For reproducibility, the unrounded eigenvalues in this example are
\[
\begin{aligned}
\lambda_{1}=\lambda_{3}=\lambda_{5}
 &=\left(C_o+\sqrt{C_o^2-1}\right)e^{11\pi i/40},
 &C_o&=\frac{\cos(\pi/4)}{\sin(3\pi/40)},\\
\lambda_{2}=\lambda_{4}=\lambda_{6}
 &=\left(C_e+\sqrt{C_e^2-1}\right)e^{-37\pi i/120},
 &C_e&=\frac{\cos(\pi/4)}{\sin(7\pi/40)}.
\end{aligned}
\]
The even rotation angle in the table uses the representative in
\((-\pi,\pi]\); the corresponding lift is used in both products.

\begin{samepage}
Multiplying the two sides separately gives, to six decimal places,
{\footnotesize
\[
\begin{aligned}
A_1A_2A_3&\simeq
\begin{pmatrix}
2.004600+7.481270i&-0.982619-3.667184i&-0.923844-0.715781i\\
-0.982619-3.667184i&-0.384363+1.297588i&0.028342+0.105772i\\
-0.923844-0.715781i&0.028342+0.105772i&0.001668+0.006224i
\end{pmatrix},\\[2mm]
A_6^{-1}A_5^{-1}A_4^{-1}&\simeq
\begin{pmatrix}
2.004600+7.481270i&-0.982619-3.667184i&-0.923844-0.715781i\\
-0.982619-3.667184i&-0.384363+1.297588i&0.028342+0.105772i\\
-0.923844-0.715781i&0.028342+0.105772i&0.001668+0.006224i
\end{pmatrix}.
\end{aligned}
\]
}
With unrounded input, the maximum entrywise difference is about
\(1.2\times10^{-14}\) in double precision. This is a direct numerical
check of the matrix identity in Theorem \ref{thm:main} for the stated example, rather than a
proof for the whole parameter family.
\end{samepage}

\end{document}